\documentclass[a4paper]{amsart}
\usepackage{graphicx}
\usepackage{amsaddr}
\usepackage{amsfonts}
\usepackage{amsmath}
\usepackage{amsthm}
\usepackage{amssymb}
\usepackage{mathtools}
\usepackage{mathrsfs}
\usepackage{tikz}
\usepackage[numbers]{natbib}
\usepackage[fit]{truncate}
\usepackage{color}
\usepackage{bm}
\usepackage{hyperref}
\hypersetup{hidelinks,pdftitle={Hitting probability for Reflected Brownian Motion at Small Target},pdfauthor={Yuchen Fan}}
\usepackage{float}
\usepackage{geometry}
\allowdisplaybreaks[4]
\usetikzlibrary{arrows.meta,positioning}

\theoremstyle{plain}
\newtheorem{theorem}{Theorem}[section]
\newtheorem{proposition}[theorem]{Proposition}
\newtheorem{lemma}[theorem]{Lemma}
\newtheorem{definition}[theorem]{Definition}
\newtheorem{remark}[theorem]{Remark}
\numberwithin{equation}{section}
\newcommand{\dd}{\,\mathrm{d}}
\renewcommand{\Cap}{\operatorname{Cap}}
\newcommand{\tg}{\mathrm{tg}}
\newcommand{\ab}{\mathrm{abs}}
\newcommand{\rf}{\mathrm{ref}}
\title[Hitting probability for reflected Brownian motion]
{Hitting probability for Reflected Brownian Motion at Small Target}
\author{Yuchen Fan}
\address{Courant Institute, New York University, USA}
\email{yf2085@cims.nyu.edu}
\thanks{Supported by Mathematical Institute, University of Oxford.}
\thanks{Corresponding author. ORCID: 0000-0002-8445-272X}
\makeatletter
\@ifundefined{subjclassname@2020}{\@namedef{subjclassname@2020}{2020 Mathematics Subject Classification}}{}
\makeatother
\date{}
\begin{document}
\begin{abstract}
We investigate the asymptotic behaviour of the hitting probability of a small target of size $O(\epsilon)$ for reflected Brownian motion in domains with sufficiently smooth boundaries. The boundary of the domain is assumed to consist of reflecting parts, absorbing parts, and the target. In this setting, we show that the order of the asymptotic hitting probability depends only on the local geometry near the target, and is comparable to the reciprocal of the fundamental solution of the negative Laplace operator. We also strengthen our result for certain domains in $\mathbb{C}$, where we obtain sharp asymptotics.
\end{abstract}

\maketitle

\section{Introduction}

Reflected Brownian motion (abbreviated as RBM) evolves as Brownian motion in the interior of a domain and is pushed in the inward normal direction at its reflecting boundary. We consider a mixed boundary model: the motion is reflected on one part of the boundary and stopped on the absorbing part, which consists of a small target and a fixed non-target part. Although this description is intuitive, constructing RBM in domains of $\mathbb{R}^n$ for $n\geq2$ requires suitable boundary assumptions \cite{Burdzy_2013,anderson76,LionsSznitman1984}.
Despite this difficulty, RBM has become an important area of research in both probability theory and statistical physics \cite{Burdzy_2013,hittingdensity,Chaigneau_2023, Bras_2023, LaknerLiuReed2023RBMwedge}. In particular, its hitting behaviour at small targets has attracted the interest of many researchers \cite{redner_2001, PhysRevE.105.054107, LaknerLiuReed2023RBMwedge, Bras_2023}. Related small-target asymptotics have also been studied for Brownian motion with drift and for structured spherical targets \cite{10.3150/20-BEJ1257,doi:10.1137/16M1077659}.

Consider a domain in $\mathbb{R}^n$. We divide its boundary into a shrinking target, a reflecting part and a fixed non-target absorbing part (Definition \ref{abstgandref}). In a simply connected planar domain with connected non-target absorbing part, these form four successive arcs: the shrinking target, a reflecting arc, the non-target absorbing arc, and another reflecting arc. In this paper, we focus on the asymptotic hitting probability of a RBM inside the domain, namely the probability that it hits the target before hitting the non-target absorbing boundary, when the target has size $O(\epsilon)$.

The narrow escape literature studies the mean time to reach a small absorbing window in an otherwise reflecting boundary; see \cite{HolcmanSchuss2014}. The related leakage problem includes both a small absorbing hole and a fixed absorbing part \cite{SingerSchussHolcman2008}. These works already identify logarithmic behaviour in two dimensions and the role of capacity in higher dimensions. Thus the dimensional orders obtained here are consistent with known small-window asymptotics. Our statements concern the splitting probability from a fixed starting point, with a fixed competing absorbing set, and give comparison bounds for the smooth domain class below. The sharp planar coefficient is obtained by a conformal calculation. Explicit computations for particular geometries provide complementary information \cite{Chaigneau_2023,Bras_2023,LaknerLiuReed2023RBMwedge}.

Instead, in this paper, we observe that RBM behaves well in a large class of domains, which we define as ``Smooth Uniform Lipschitz Domains'' (which we will abbreviate as SULDs, compare Definition \ref{Nicedomain}). The name SULD is terminology used in this paper for the uniform smooth boundary assumptions below; the existence of RBM follows from established theory \cite{anderson76,LionsSznitman1984}. In this class of domains, the order of the asymptotic hitting probability only depends on the local geometry near the target; the comparison constants may depend on the domain, the target centre, the fixed absorbing boundary and the starting point. When the target is very small, the local geometry of a SULD near the target is close to a half-space, and the asymptotic hitting probability becomes comparable to the reciprocal of Newtonian potential, which is a fundamental solution of $-\Delta u=\delta$. We also use methods from complex analysis, which have been widely applied in the literature \cite{lawler08,oksendal1981brownian}, to establish a more precise result for certain SULDs in $\mathbb{C}$ (Theorem \ref{strongerresult}).  

The present paper concerns hitting probabilities, rather than mean hitting times. Many mean hitting time asymptotics are known in the narrow escape setting \cite{HolcmanSchuss2014}, including results for elongated domains \cite{PhysRevE.105.054107}. Extending our comparison argument to rough boundaries, such as fractal boundaries, requires additional analysis. RBM itself has been constructed in substantially more general domains \cite{BassHsu1991,Burdzy_2013}, but the smooth boundary charts used here are unavailable in that setting.

\noindent\textbf{Outline.} In Section~2, we define SULDs and recall the construction of RBM. Section~3 introduces the target, reflecting boundaries, absorbing boundaries, and states our main results (Theorem~\ref{result} and Theorem~\ref{strongerresult}). Section~3.3 explains the proof strategy. In Section~4, we prove Theorem~\ref{strongerresult} using conformal maps. Finally, Section~5 establishes Theorem~\ref{result} in general dimensions.

\section*{Acknowledgements}
The author wants to express his deepest gratitude to Prof. Dmitry Belyaev as his supervisor of this project. The author would also like to thank two anonymous referees for their valuable comments. The author has used AI tool (Chatgpt 6.1 sol) to polish this version of draft. In particular, the author has inspected serval errors using this AI tool and corrected them, and to sketched three more accurate figures. The author retains full responsibility for the final manuscript, including its mathematical claims, proofs, citations, and conclusions.

\section{Domains and Construction of RBM}
\subsection{Construction of RBM on SULDs}
In this subsection, we recall the construction of RBM on connected domains in $\mathbb{R}^n$ with suitable smooth and analytical conditions. We follow the construction given in \cite{anderson76}, applying existing existence and uniqueness results rather than refining them; see also \cite{LionsSznitman1984}. Constructions in more general domains can be given, for example, in \cite{semimartingale,BassHsu1991} using Dirichlet forms.

\begin{definition}\label{Nicedomain}
Let $\Omega\subset\mathbb{R}^n$, $n\geq2$, be a non-empty, open, connected smooth domain. We say it is a SULD if there exists
\begin{enumerate}
\item a countable open cover $\mathscr U=\{U^0,U^1,\ldots\}$ of $\overline\Omega$, where $U^0\subseteq\Omega$ is equipped with the original Euclidean coordinates of $\mathbb R^n$, and the boundary neighbourhoods $U^k$, $k\geq1$, are open balls;
\item a family of $C^3$ coordinate diffeomorphisms $\Phi_k=(\Phi_k^1,\ldots,\Phi_k^n):U^k\to V^k$, where $V^k$ is open in $\mathbb R^n$, such that
\[
\Phi_k(U^k\cap\Omega)=V^k\cap\{y_1>0\},\qquad
\Phi_k(U^k\cap\partial\Omega)=V^k\cap\{y_1=0\},
\]
and $D\Phi_k(x)\mathbf n(x)=a_k(x)e_1$ with $a_k(x)>0$ for $x\in U^k\cap\partial\Omega$, where $\mathbf n$ denotes the inward unit normal;
\end{enumerate}
such that
\[
0<\inf_{\substack{k\geq1\\x,y\in U^k,\ x\ne y}}
\frac{|\Phi_k(x)-\Phi_k(y)|}{|x-y|}
\leq\sup_{\substack{k\geq1\\x,y\in U^k,\ x\ne y}}
\frac{|\Phi_k(x)-\Phi_k(y)|}{|x-y|}<\infty,
\]
and
\[
\sup_{i,k}\|\Phi_k^i\|_{W^{3,\infty}(U^k)}<\infty.
\]
The uniform boundary regularity also includes a common chart radius: there is $r_*>0$ such that every $p\in\partial\Omega$ has a chart of this type on $B(p,r_*)$, with the same bounds. Charts may be translated so that $\Phi_k(p)=0$.
\end{definition}

Intuitively, these maps straighten the boundary into a hyperplane, with the domain on its positive side; their uniform bounds prevent the boundary charts from degenerating. The boundary-to-hyperplane condition is imposed explicitly and is not inferred just from the condition on normals. Many standard domains are SULDs. For example, it could be easily shown that smooth bounded domains in $\mathbb R^n$ and analytic domains in $\mathbb C$ (compare Section~2.2) are all SULDs. We now recall the construction of RBM in SULDs following \cite{anderson76}. Figure~\ref{localchart} illustrates a chart.

\begin{figure}[htbp]
\centering
\begin{tikzpicture}[scale=0.78,>=Stealth,font=\small]
 \fill[gray!10] (-2,0.15) .. controls (-0.9,-0.45) and (0.4,0.45) .. (1.6,0.1) -- (1.6,1.8) -- (-2,1.8) -- cycle;
 \draw[thick] (-2,0.15) .. controls (-0.9,-0.45) and (0.4,0.45) .. (1.6,0.1);
 \node at (-0.3,1.25) {$\Omega\cap U^k$};
 \fill (-0.5,0) circle (1.7pt);
 \draw[->] (-0.5,0)--(-0.5,0.75) node[left] {$\mathbf n$};
 \node[below] at (-0.5,-0.05) {$p$};
 \draw[->,thick] (2,0.8)--(3.6,0.8) node[midway,above] {$\Phi_k$};
 \begin{scope}[xshift=6cm]
  \fill[gray!10] (-1.6,0) rectangle (1.6,1.8);
  \draw[thick] (-1.6,0)--(1.6,0) node[right] {$y_1=0$};
  \draw[->] (0,0)--(0,0.8) node[right] {$e_1$};
  \fill (0,0) circle (1.7pt);
  \node[below] at (0,0) {$0$};
  \node at (0.1,1.3) {$y_1>0$};
 \end{scope}
\end{tikzpicture}
\caption{A boundary chart sends the boundary to $y_1=0$ and the interior to $y_1>0$. The vertical direction on the right represents the first coordinate.}
\label{localchart}
\end{figure}

\begin{definition}\label{SULDconstruction}
Let $\Omega$ be a SULD, and $\gamma(x)$ be a vector field varying smoothly on the boundary and uniformly bounded away from the tangent space of the boundary at $x$. We require that $\gamma$ points inward:
\[
\gamma(x)\cdot\mathbf n(x)\geq c_0>0,\qquad x\in\partial\Omega,
\]
and take $\gamma$ to be bounded and $C^2$ with bounded derivatives. Then the stochastic differential equation
\begin{equation}\label{eq1'}
\dd X_t^x=\dd W_t+\mathbf1_{\partial\Omega}(X_t^x)\gamma(X_t^x)\dd\xi_t^x,
\qquad X_0^x=x\in\overline\Omega,\quad\xi_0^x=0,
\end{equation}
has a pathwise unique strong solution pair $(X_t^x,\xi_t^x)$, with $X_t^x\in\overline\Omega$, if we assume that $\xi_t^x$ is continuous, non-decreasing in $t$, and increases only on the Lebesgue null set $\Delta=\{t:X_t^x\in\partial\Omega\}$. In particular, when $\gamma$ is the inward unit normal vector at each $x$, we define the RBM in $\Omega$ (starting at $x$) to be the corresponding solution process $X_t^x$. This is the standard construction recalled from \cite{anderson76,LionsSznitman1984}.
\end{definition}

A general boundary chart does not send RBM to standard RBM in a half-space. Indeed, It\^o's formula gives, up to exit from the chart,
\begin{equation}\label{transformedSDE}
 \dd\Phi_k(X_t)=D\Phi_k(X_t)\dd W_t
 +\tfrac12\Delta\Phi_k(X_t)\dd t
 +D\Phi_k(X_t)\mathbf n(X_t)\dd\xi_t.
\end{equation}
The resulting reflected diffusion has variable, uniformly elliptic coefficients. Section~5 uses comparison of Dirichlet energies under these charts, rather than identifying this process with standard Brownian motion.

\begin{remark}
\label{Upperhalfplaneequivalent} By a direct application of Tanaka's formula, if
\[
W^x_t=(W^x_1(t),W^x_2(t),\cdots,W^x_n(t))
\]
is a Brownian motion on $\mathbb{R}^n$ with $W^x_0=x\in\overline{\mathbb H_n^k}$, then
\[
Y^x_t=(W^x_1(t),W^x_2(t),\cdots,|W^x_k(t)|,\cdots,W^x_n(t))
\]
is a RBM on $\mathbb{H}^k_n=\{(x_1,x_2,\cdots,x_n)\in\mathbb{R}^n:x_k>0\}$ starting at $x$.
\end{remark}

\subsection{Equivalent Construction of RBM in Domains in \texorpdfstring{$\mathbb{C}$}{C}}
We give an equivalent construction of (complex) RBM in simply connected analytic domains of $\mathbb C$ using conformal maps as a special case. Write $\mathbb D=\{z\in\mathbb C:|z|<1\}$ and $\mathbb H=\{z\in\mathbb C:\Im z>0\}$. An \emph{analytic domain} in $\mathbb C$ is a domain where its boundary is the image of the unit circle $\partial\mathbb D$ under a univalent function on a neighbourhood of $\partial\mathbb D$; here we take the bounded component enclosed by this analytic Jordan curve. Let $W_t^x=X_t^x+iY_t^x$ be the complex Brownian motion starting from $x\in\mathbb H$, then we define RBM in the upper half-plane $\mathbb H$, $B_t^x$, by $B_t^x=X_t^x+i|Y_t^x|$. By Remark~\ref{Upperhalfplaneequivalent} it is equivalent to the SDE definition. Then we define the RBM in a general simply connected analytic domain $\Omega$ by conformal maps, where the Riemann mapping theorem ensures there is a conformal map between $\mathbb H$ and $\Omega$. More explicitly, let
\[
f(z)=i\frac{z-i}{z+i},
\]
then $f$ maps $\mathbb H$ conformally bijectively onto $\mathbb D$. Since the images of a Brownian motion under conformal maps are time-changed Brownian motions and conformal maps preserve normal reflection, $f(B_t^x)$ is a time-changed RBM in $\mathbb D$. Conversely, $f^{-1}$ has its only pole on $\overline{\mathbb D}$ at $i$. Starting from any interior point of $\mathbb D$, RBM almost surely never hits that fixed boundary point: single boundary points of a smooth planar reflecting boundary have zero capacity and are polar \cite{FukushimaOshimaTakeda2011}. Thus the inverse transformation is well-defined, by first stopping away from $i$ and then taking the limit. Zero harmonic measure for the first exit of ordinary Brownian motion alone would not justify this assertion for RBM. No exclusion of the starting point $0\in\mathbb D$ is needed; in fact $f^{-1}(0)=i\in\mathbb H$.

Now let $\Omega$ be a simply connected analytic domain and let $g:\Omega\to\mathbb D$ be a conformal bijection. Analytic boundary extension (\cite[Theorem~2.25]{conformalmapsandgeo}) extends $g^{-1}$ conformally across $\partial\mathbb D$, with non-vanishing derivative. Consequently we have the following equivalent construction.
\begin{proposition}\label{conformalRBM}
Suppose $B_t$ is the complex RBM in $\mathbb{D}$ and let $\Omega$ be an analytic simply connected domain. Let $g$ be a conformal bijection from $\Omega$ to $\mathbb{D}$. Let
\[
\sigma(t)=\inf \left\{s\geq0:\int_0^s\left|\left(g^{-1}\right)'(B_u)\right|^2\dd u>t\right\}.
\]
Then the process $\widetilde B_t=g^{-1}(B_{\sigma(t)})$ is the RBM in $\Omega$.
\end{proposition}
Given two analytic simply connected domains, then we can map the RBM from one domain to the RBM on another domain by conformal maps and time change. The time change does not affect which absorbing part is hit first.

\section{Main Result}
\subsection{Description of the Hitting Target}
The boundary of our domain will be divided into two parts: The absorbing part and the reflecting part. When the RBM hits the absorbing part, it is stopped; when it hits the reflecting part, it is then (normally) reflected. The absorbing part will also be divided into the target part and the non-target part, where we are interested in the probability that the RBM hits the target part.

\begin{definition}
\label{abstgandref} (Absorbing Boundary, Target and Reflecting Boundary) Let $\Omega$ be a SULD in $\mathcal{S}=\mathbb{R}^n$, $n\geq2$, or $\mathcal{S}=\mathbb{C}$. Let $P\in \partial\Omega$ and $B_R(P):=\{z\in \mathcal{S}:|z-P|<R\}$. Let $R>0$ be a fixed constant such that the closure of $\left(\mathcal{S}\setminus \overline{B_R(P)}\right)\cap\partial\Omega$ is non-empty, which will be defined as the \emph{non-target} part of the absorbing boundary and will be denoted by $\mathrm{abs}$. 

Let $0<\epsilon<R$ be a small constant and we define our \emph{target} part of the absorbing boundary as the closure of $B_\epsilon(P)\cap\partial\Omega$, which will be denoted by $\mathrm{tg}$. 

Let our \emph{reflecting boundary} be $\partial\Omega\setminus(\mathrm{abs}\cup\mathrm{tg})$ which will be denoted by $\mathrm{ref}$, where the RBM will experience a normal reflection once it hits this part of the boundary.
\end{definition}
We also write $\tg_\epsilon=\mathrm{tg}$ to indicate the dependence of the target on $\epsilon$.

\begin{figure}[htbp]
\centering
\begin{tikzpicture}[scale=0.85,font=\small]
 \fill[gray!7] (0,0) circle (1.65);
 \draw[thick] (0,0) circle (1.65);
 \draw[black,line width=2pt] (0,0) ++(150:1.65) arc (150:390:1.65);
 \draw[blue!65!black,line width=2.5pt] (0,0) ++(82:1.65) arc (82:98:1.65);
 \draw[dashed] (0,1.65) circle (0.23);
 \draw[dashed] (0,1.65) circle (1.65);
 \fill (0,1.65) circle (1.4pt) node[above right] {$P$};
 \node at (0,0) {$\Omega$};
 \node at (2.3,-0.8) {$\ab$};
 \draw (1.8,-0.7)--(1.43,-0.8);
 \node at (1.25,1.45) {$\rf$};
 \node at (-1.25,1.45) {$\rf$};
 \node at (-0.8,2.25) {$\tg_\epsilon$};
 \draw (-0.5,2.1)--(-0.12,1.7);
 \node at (2.0,2.65) {$B_R(P)$};
 \node at (0.8,2.1) {$B_\epsilon(P)$};
\end{tikzpicture}
\caption{Absorbing boundary, reflecting boundary and target. The two dashed circles have radii $\epsilon$ and $R$. The thick black arc is the non-target absorbing part.}
\label{target}
\end{figure}

\begin{definition}
Let $B^{x_0}_t$ be the RBM in $\Omega$. Let $\tau=\inf\{t\geq 0: B^{x_0}_t\in\mathrm{abs}\cup\mathrm{tg}\}$ be the stopping time that RBM hits the absorbing part of the boundary. We consider the stopped process $B^{x_0}_{\tau\wedge t}$ to be the associated stopped RBM. In the later texts, we use $B^{x_0}_t$ to denote $B^{x_0}_{\tau\wedge t}$ if there is no ambiguity.
For the underlying unstopped reflected motion, let
\begin{gather*}
 T_\epsilon=\inf\{t\geq0:B_t^{x_0}\in\tg_\epsilon\},\qquad
 T_{\ab}=\inf\{t\geq0:B_t^{x_0}\in\ab\},\\
 \tau=T_\epsilon\wedge T_{\ab}.
\end{gather*}
We write
\[
 u_\epsilon(x_0)=\mathbb P_{x_0}(T_\epsilon<T_{\ab})
 =\mathbb P_{x_0}(\tau<\infty,\ B_\tau\in\tg_\epsilon).
\]
The explicit condition $\tau<\infty$ covers unbounded domains in which escape without absorption may occur.
\end{definition}

\subsection{Main Result}
The first theorem gives the dimensional order as the target shrinks. The constants are uniform in $\epsilon$, for a fixed domain, target centre, competing absorbing boundary and starting point.
\begin{theorem}\label{result}
Let $\Omega$ be a SULD in $\mathbb R^n$, $n\geq2$, and let the boundary parts be as in Definition~\ref{abstgandref}. For $x_0\in\Omega$, put
\[
 G_{n,y}(x)=
 \begin{cases}
 \displaystyle\frac{|x-y|^{2-n}}{(n-2)\omega_{n-1}},&n\geq3,\\[4pt]
 \displaystyle-\frac1{2\pi}\log|x-y|,&n=2,
 \end{cases}
\]
where $\omega_{n-1}=|\mathbb S^{n-1}|$ is the surface area of the unit sphere. Thus $-\Delta G_{n,y}=\delta_y$ in the distributional sense. There exist positive constants $C_1,C_2$, depending on $\Omega,P,R,x_0$, such that
\begin{equation}\label{comparison}
 \frac{C_1}{G_{n,0}(\bm\epsilon)}
 \leq u_\epsilon(x_0)
 \leq\frac{C_2}{G_{n,0}(\bm\epsilon)},
 \qquad \bm\epsilon=(\epsilon,0,\ldots,0),
\end{equation}
for all sufficiently small $\epsilon$. Equivalently,
\[
 u_\epsilon(x_0)\asymp
 \begin{cases}
 \epsilon^{n-2},&n\geq3,\\
 (\log(1/\epsilon))^{-1},&n=2.
 \end{cases}
\]
\end{theorem}
There is no extra factor $n$ in the denominator with this definition of $\omega_{n-1}$. If the volume of the unit ball is denoted by $\omega_n$, then $|\mathbb S^{n-1}|=n\omega_n$, and the equivalent denominator is $n(n-2)\omega_n$.

The logarithmic comparison becomes a limit in the following planar setting. The connectedness assumption refers to the \emph{non-target} absorbing set $\ab$; the full absorbing boundary also includes the disjoint target.
\begin{theorem}\label{strongerresult}
Let $\Omega$ be a simply connected analytic domain of $\mathbb C$. If $\ab$ is connected, then for every $z_0\in\Omega$ there is a constant $C(\Omega,P,R,z_0)<0$ such that
\begin{equation}\label{planarlimit}
 \lim_{\epsilon\downarrow0}u_\epsilon(z_0)\log\epsilon
 =C(\Omega,P,R,z_0).
\end{equation}
In particular, $u_\epsilon(z_0)\log(1/\epsilon)$ has a strictly positive limit.
\end{theorem}
Unlike the dimensional order in Theorem~\ref{result}, this coefficient records the global geometry and the location of the competing absorbing arc. Section~4 identifies it through a conformal map.

\subsection{Proof strategy}
In the plane we map the domain to the upper half-plane, with the fixed absorbing arc on the negative real axis and the target near $1$. A Schwarz--Christoffel map then turns the mixed boundary problem into a rectangle with two absorbing and two reflecting sides. The hitting probability is linear across the rectangle; its long side grows like $\log(1/\epsilon)$.

In general dimensions we first study an excursion near the target, stopping it when it leaves a fixed boundary neighbourhood. The local hitting probability is comparable to the relative capacity of the target. A smooth boundary chart compares this capacity with that of a flat disk, which is of order $\epsilon^{n-2}$ for $n\geq3$ and $1/\log(1/\epsilon)$ for $n=2$. A fixed positive chance of reaching $\ab$ between successive visits controls the number of returns to the neighbourhood. This contributes constants independent of $\epsilon$.

\section{Proof of the Strengthened Result in \texorpdfstring{$\mathbb{C}$}{C}}
\subsection{Explicit Computation in \texorpdfstring{$\mathbb{H}$}{H}}
\begin{theorem}\label{theoreminH}
Let $\mathbb{H}=\{z\in\mathbb{C}:\Im(z)>0\}$ be our domain and $(0,1)\cup(1+\epsilon,\infty)$ be our reflecting boundary, and $(-\infty,0]\cup[1,1+\epsilon]$ be our absorbing boundary while $[1,1+\epsilon]$ is our target. Let $(B_t,t\geq0)$ be the RBM in $\mathbb{H}$ starting at $z_0\in\mathbb H$. Then
\begin{equation}\label{halfplaneasymptotic}
\mathbb P\left(B^{z_0}_\tau\in(1,1+\epsilon)\right)
\sim\frac{C(z_0)}{\log\epsilon},\qquad \epsilon\downarrow0,
\end{equation}
where $C(z_0)=\Re\left(\lim_{\epsilon\downarrow0}f_\epsilon(z_0)\right)<0$ and $f_\epsilon$ is defined in \eqref{SCmap}. Equivalently, the probability is asymptotic to $H(z_0)/\log(1/\epsilon)$, where
\[
H(z)=\log\left|\frac{1+\sqrt z}{1-\sqrt z}\right|>0
\]
and $\sqrt z$ has argument in $(0,\pi/2)$ for $z\in\mathbb H$.
\end{theorem}
\begin{proof}
The Schwarz--Christoffel mapping formula gives
\begin{equation}\label{SCmap}
 f_\epsilon(z)=\int_0^z
 \frac{\dd\xi}{\sqrt\xi\sqrt{\xi-1}\sqrt{\xi-(1+\epsilon)}}.
\end{equation}
Each square root is the analytic branch in $\mathbb H$ obtained from arguments in $(0,\pi)$, and the integral from $0$ is its boundary limit. The map sends $\mathbb H$ onto a rectangle with vertices
\begin{align*}
 A&=f_\epsilon(0)=0,& B&=f_\epsilon(1)=-L_\epsilon,\\
 C&=f_\epsilon(1+\epsilon)=-L_\epsilon-il_\epsilon,&
 D&=f_\epsilon(\infty)=-il_\epsilon,
\end{align*}
where the positive side lengths are
\begin{align}
 L_\epsilon&=\int_0^1
 \frac{\dd\xi}{\sqrt\xi\sqrt{1-\xi}\sqrt{1+\epsilon-\xi}},\label{Ldef}\\
 l_\epsilon&=\int_1^{1+\epsilon}
 \frac{\dd\xi}{\sqrt\xi\sqrt{\xi-1}\sqrt{1+\epsilon-\xi}}.\nonumber
\end{align}
By conformal invariance up to time change, the reflecting sides are $AB$ and $CD$, the absorbing sides are $AD$ and $BC$, and the target is $BC$.

In the rectangle the function $v(w)=-\Re w/L_\epsilon$ is harmonic, equals $0$ on $AD$ and $1$ on $BC$, and has zero normal derivative on $AB\cup CD$. It is the unique weak solution of this mixed problem: the difference $h$ of two solutions has zero trace on the absorbing sides, and integration by parts gives $\int|\nabla h|^2=0$, hence $h=0$. Equivalently, the horizontal coordinate of the reflected motion is one-dimensional Brownian motion stopped at $-L_\epsilon$ or $0$. Optional stopping therefore gives the exact formula
\begin{equation}\label{rectangleprob}
 \mathbb P_{z_0}(T_{[1,1+\epsilon]}<T_{(-\infty,0]})
 =-\frac{\Re f_\epsilon(z_0)}{L_\epsilon}.
\end{equation}
This also follows by unfolding the reflecting sides into a vertical strip; uniqueness there holds in the class of bounded harmonic functions with the specified boundary data.

Define the complete elliptic integral of the first kind, with modulus $k$, by
\[
 K(k)=\int_0^{\pi/2}(1-k^2\sin^2\theta)^{-1/2}\dd\theta,
 \qquad 0\leq k<1.
\]
Substituting $\xi=\sin^2\theta$ in \eqref{Ldef} yields
\[
 L_\epsilon=2\int_0^{\pi/2}\frac{\dd\theta}{\sqrt{1+\epsilon-\sin^2\theta}}
 =\frac2{\sqrt{1+\epsilon}}
 K\!\left(\frac1{\sqrt{1+\epsilon}}\right).
\]
The expansion $K(k)=\log(4/\sqrt{1-k^2})+
O((1-k^2)\log(1/(1-k^2)))$ as $k\uparrow1$
\cite[\S19.12]{DLMF} gives
\begin{equation}\label{Lexpansion}
 L_\epsilon=\log\frac{16}{\epsilon}
 +O\!\left(\epsilon\log\frac1\epsilon\right).
\end{equation}
On compact subsets of $\mathbb H$, dominated convergence in \eqref{SCmap} gives
\[
 f_\epsilon(z)\longrightarrow F(z):=
 \int_0^z\frac{\dd\xi}{\sqrt\xi(\xi-1)}.
\]
Indeed, the path can be chosen to avoid $1$, and the singularity at $0$ is integrable uniformly in $\epsilon$. With $w=\sqrt\xi$,
\[
 F(z)=\log\frac{1-\sqrt z}{1+\sqrt z},\qquad
 \Re F(z)=\log\left|\frac{1-\sqrt z}{1+\sqrt z}\right|=-H(z),
\]
where the logarithm is continued from value $0$ at $z=0$. Since $\Re\sqrt z>0$, $H(z)>0$. Equations~\eqref{rectangleprob} and~\eqref{Lexpansion} prove \eqref{halfplaneasymptotic}. The convergence is locally uniform in $z_0\in\mathbb H$.
\end{proof}

\subsection{Proof of Theorem~\ref{strongerresult}}
We first record why a sufficiently small target is a single boundary arc.
\begin{lemma}\label{lemmaonecomponent}
Let $\Omega$ be a simply connected analytic domain and $P\in\partial\Omega$. For all sufficiently small $\epsilon>0$, $B_\epsilon(P)\cap\partial\Omega$ has one connected component, and $\partial B_\epsilon(P)$ meets $\partial\Omega$ at two points.
\end{lemma}
\begin{proof}
Choose an arclength parametrisation $\gamma(s)$ near $P=\gamma(0)$. Smoothness and $|\gamma'(0)|=1$ give
\[
 \frac{\mathrm d}{\mathrm ds}|\gamma(s)-P|^2=2s+O(s^2).
\]
Thus the distance from $P$ is strictly increasing on the positive side and strictly decreasing on the negative side near $0$. The remainder of the Jordan curve, outside a fixed parameter neighbourhood of $0$, is compact and has positive distance from $P$. A sufficiently small circle therefore cuts the boundary exactly once on each side of $P$, proving both assertions.
\end{proof}
\begin{proof}[Proof of Theorem~\ref{strongerresult}]
Since $\ab$ is a connected, non-empty proper closed subset of the analytic Jordan boundary and has non-empty relative interior, it is a closed arc with distinct endpoints $B,C$. Its complementary open arc contains $P$. A Riemann map $f:\Omega\to\mathbb D$ extends analytically across the boundary with non-vanishing derivative. Composing with a disk-to-half-plane M\"obius map $M$, we obtain the conformal bijection
\[
 h=M\circ f:\Omega\to\mathbb H,\qquad
 h(B)=0,\quad h(P)=1,\quad h(C)=\infty,
\]
where the labelling of $B,C$ is chosen in boundary order. Then $h(\ab)=(-\infty,0]$ on the extended real boundary. Figure~\ref{fig4} indicates the boundary correspondence.

\begin{figure}[htbp]
\centering
\begin{tikzpicture}[scale=0.8,>=Stealth,font=\small]
 \draw[thick] (0,0) circle (1.25);
 \draw[line width=2pt] (0,0) ++(160:1.25) arc (160:380:1.25);
 \draw[blue!65!black,line width=2.5pt] (0,0) ++(82:1.25) arc (82:98:1.25);
 \node at (0,0) {$\Omega$};
 \node[above] at (0,1.25) {$P$};
 \node[left] at (-1.2,0.43) {$B$};
 \node[right] at (1.2,0.43) {$C$};
 \draw[->,thick] (1.8,0.4)--(3.4,0.4) node[midway,above] {$h$};
 \begin{scope}[xshift=5.5cm]
  \draw[thick] (-1.6,0)--(2.4,0);
  \draw[line width=2pt] (-1.6,0)--(0,0);
  \draw[blue!65!black,line width=2.5pt] (0.9,0)--(1.3,0);
  \fill (0,0) circle (1.2pt) node[below] {$0$};
  \fill (1.1,0) circle (1.2pt) node[below=5pt] {$1$};
  \node at (0.5,1.0) {$\mathbb H$};
  \node[above] at (-0.9,0) {$\ab$};
  \node[above] at (1.1,0) {$h(\tg_\epsilon)$};
  \draw[->] (2.5,0.5)--(3.1,0.5) node[right] {$h(C)=\infty$};
 \end{scope}
\end{tikzpicture}
\caption{The non-target absorbing arc maps to the negative real half-line and the target maps to an interval shrinking to $1$.}
\label{fig4}
\end{figure}

By Lemma~\ref{lemmaonecomponent}, the target has endpoints $A_\epsilon,D_\epsilon$, with $|A_\epsilon-P|=|D_\epsilon-P|=\epsilon$. Since $h$ extends conformally near $P$, its image is
\[
 h(\tg_\epsilon)=[1-k_1(\epsilon),1+k_2(\epsilon)],
 \qquad k_i(\epsilon)=|h'(P)|\epsilon+o(\epsilon),\quad i=1,2.
\]
This follows by applying the first-order expansion of $h$ to each endpoint separately. In particular, both $k_i$ are positive for sufficiently small $\epsilon$; their sum, rather than their difference, is the target width.

Set $a_\epsilon=1-k_1(\epsilon)$ and
\[
 \delta_\epsilon=\frac{k_1(\epsilon)+k_2(\epsilon)}{1-k_1(\epsilon)}
 \sim2|h'(P)|\epsilon,\qquad
 z_\epsilon=\frac{h(z_0)}{a_\epsilon}\longrightarrow h(z_0).
\]
The map $z\mapsto h(z)/a_\epsilon$ preserves the non-target absorbing half-line and sends the target to $[1,1+\delta_\epsilon]$. Conformal invariance and the locally uniform limit in Theorem~\ref{theoreminH} imply
\[
 u_\epsilon(z_0)\log(1/\epsilon)
 \longrightarrow H(h(z_0)),
\]
since $\log(1/\delta_\epsilon)=\log(1/\epsilon)+O(1)$. Hence
\begin{equation}\label{correctcoefficient}
 C(\Omega,P,R,z_0)=-H(h(z_0))
 =\log\left|\frac{1-\sqrt{h(z_0)}}{1+\sqrt{h(z_0)}}\right|<0,
\end{equation}
which proves the theorem.
\end{proof}

\section{Proof of the Result in General SULD Domains}
In this section we turn to prove Theorem \ref{result}. When $\epsilon$ is very small, we first localize our domain. We compare Dirichlet energies under a boundary chart: a non-conformal chart does not preserve standard Brownian motion, as \eqref{transformedSDE} shows.

\subsection{Localization of Domain}
Choose a boundary chart $\Phi=\Phi_1$, with $\Phi(P)=0$, and radii $0<r_2<r_1$ so small that the closed chart half-ball lies away from $\ab$, and $x_0$ lies outside it. Write
\begin{gather*}
 B_r^+=B_r(0)\cap\{y_1>0\},\qquad
 U'=\Phi^{-1}(B_{r_1}^+),\\
 \mathscr C_j=\Phi^{-1}(\partial B_{r_j}(0)\cap\{y_1\geq0\}),\quad j=1,2.
\end{gather*}
All sets are taken within the fixed chart. For small $\epsilon$, $\tg_\epsilon$ lies inside the inner half-ball. Let
\[
 S_1=\inf\{t\geq0:B_t\in\mathscr C_1\},\qquad
 q_\epsilon(x)=\mathbb P_x(T_\epsilon<S_1),\quad x\in\overline{U'}.
\]
This is the local target hitting probability, with killing on $\mathscr C_1$ and reflection on the physical boundary outside the target. Define
\[
 b_n(\epsilon)=
 \begin{cases}\epsilon^{n-2},&n\geq3,\\(\log(1/\epsilon))^{-1},&n=2.\end{cases}
\]
We will show $q_\epsilon(x)\asymp b_n(\epsilon)$ uniformly for $x\in\mathscr C_2$, and then control excursions that leave $U'$.

\subsection{Explicit Computation and Estimation in \texorpdfstring{$U'$}{U'}}
For a compact boundary target $K$ separated from $\mathscr C_1$, define its relative capacity by
\begin{equation}\label{relativecapacity}
 \Cap_{U'}(K)=\inf_v\int_{U'}|\nabla v|^2\dd x,
\end{equation}
where $v\in H^1(U')$, its trace is $0$ on $\mathscr C_1$, and its trace is at least $1$ on $K$. Trace conditions on compact sets are understood quasi-everywhere (outside a set of capacity zero). The remaining boundary has the natural Neumann condition. This is the usual equilibrium-potential construction for the Dirichlet form with killing on $\mathscr C_1$ \cite[Chapters~2 and~4]{FukushimaOshimaTakeda2011}.

\begin{lemma}\label{localestimate}
There are constants $c,C>0$, independent of $\epsilon$, such that
\[
 c\,b_n(\epsilon)\leq q_\epsilon(x)\leq C\,b_n(\epsilon),
 \qquad x\in\mathscr C_2,
\]
for all sufficiently small $\epsilon$.
\end{lemma}
\begin{proof}
The equilibrium potential of $K=\tg_\epsilon$ is $q_\epsilon$. Let $\mathcal G_{U'}$ be the Green kernel for $-\Delta$ in $U'$ with Dirichlet condition on $\mathscr C_1$ and Neumann condition on the physical boundary. The equilibrium measure $\mu_\epsilon$ satisfies
\begin{equation}\label{equilibriumrepresentation}
 q_\epsilon(x)=\int_{\tg_\epsilon}\mathcal G_{U'}(x,y)\dd\mu_\epsilon(y),
 \qquad \mu_\epsilon(\tg_\epsilon)=\Cap_{U'}(\tg_\epsilon).
\end{equation}
The Green kernel extends continuously to the reflecting boundary away from its pole. It is strictly positive there, by the maximum principle and the boundary point lemma. These are local smooth-boundary properties; reflecting Green kernels and boundary regularity are also treated in \cite{BassHsu1991}. Fix $\epsilon_0$ so small that $\tg_{\epsilon_0}$ and $\mathscr C_2$ are disjoint and both avoid the killing boundary. Compactness gives positive finite bounds for $\mathcal G_{U'}$ on $\mathscr C_2\times\tg_{\epsilon_0}$. Thus
\begin{equation}\label{qcapacity}
 q_\epsilon(x)\asymp\Cap_{U'}(\tg_\epsilon),\qquad x\in\mathscr C_2.
\end{equation}

The bi-Lipschitz chart gives constants $k_1,k_2>0$ with
\[
 D_{k_1\epsilon}\subseteq\Phi(\tg_\epsilon)\subseteq D_{k_2\epsilon},
 \qquad D_a=\{y:y_1=0,\ |y|\leq a\}.
\]
For $w=v\circ\Phi^{-1}$, change of variables gives
\begin{equation}\label{energycomparison}
 \int_{U'}|\nabla v|^2\dd x
 =\int_{B_{r_1}^+}(\nabla w)^\mathsf T A(y)\nabla w\dd y,
 \quad
 A=|\det D\Phi^{-1}|\,D\Phi\,D\Phi^\mathsf T,
\end{equation}
with the derivatives on the right evaluated at $\Phi^{-1}(y)$. Uniform bi-Lipschitz bounds imply $m_0 I\leq A(y)\leq M_0 I$ for fixed $m_0,M_0>0$. Taking infima in \eqref{energycomparison} and using monotonicity in the target compares \eqref{relativecapacity} with the Euclidean half-ball capacities of $D_{k_i\epsilon}$. By even reflection, the latter are exactly half the full-ball condenser capacities $\Cap_{B_{r_1}}(D_{k_i\epsilon})$.

For completeness, these flat capacities have the orders
\begin{equation}\label{flatcapacity}
 \Cap_{B_{r_1}}(D_a)\asymp
 \begin{cases}a^{n-2},&n\geq3,\\(\log(1/a))^{-1},&n=2.\end{cases}
\end{equation}
One way to see this is the reciprocal Green-energy formula for condenser capacity \cite[Chapter~2]{FukushimaOshimaTakeda2011}:
\[
 \Cap_{B_{r_1}}(D_a)^{-1}
 =\inf_{\nu\in\mathcal P(D_a)}
 \iint g_{r_1}(x,y)\dd\nu(x)\dd\nu(y),
\]
where $g_{r_1}$ is the Dirichlet Green kernel for $-\Delta$ in the full ball and $\mathcal P(D_a)$ is the set of probability measures on $D_a$. For $n\geq3$, uniformly near the origin,
$g_{r_1}(x,y)=c_n|x-y|^{2-n}+O(1)$. Scaling $D_a=aD_1$ gives minimum energy of order $a^{2-n}$: the lower bound follows from $|x-y|\leq2a$, and normalized $(n-1)$-dimensional area measure on the disk gives the finite upper bound. For $n=2$, $D_a$ is a segment and
\[
 g_{r_1}(x,y)=\frac1{2\pi}\log\frac1{|x-y|}+O(1).
\]
The same scaling shows that its minimum energy is
$(2\pi)^{-1}\log(1/a)+O(1)$; the logarithmic singularity is integrable against normalized length measure. This proves \eqref{flatcapacity}, and then \eqref{qcapacity} proves the lemma.
\end{proof}

We also record the radial calculation behind the same dimensional orders. For ordinary Brownian motion in the full annulus $a<|y|<r_1$, the probability of hitting the inner sphere first is the unique solution of
\[
 \Delta v=0,\qquad v|_{|y|=a}=1,\qquad v|_{|y|=r_1}=0.
\]
Uniqueness follows from the maximum principle on the bounded annulus. The solution can be written using the same fundamental solution as in Theorem~\ref{result}:
\[
v(y)=\frac{G_{n,0}(y)-G_{n,0}(r_1e_1)}
{G_{n,0}(ae_1)-G_{n,0}(r_1e_1)}.
\]
At $|y|=r_2$ the solution is
\begin{equation}\label{annulusprob}
 v(r_2)=
 \begin{cases}
 \displaystyle\frac{r_1^{2-n}-r_2^{2-n}}{r_1^{2-n}-a^{2-n}},&n\geq3,\\[5pt]
 \displaystyle\frac{\log(r_2)-\log(r_1)}{\log(a)-\log(r_1)},&n=2.
 \end{cases}
\end{equation}
This verifies the orders explicitly. The capacity argument is needed to handle a boundary disk and the variable coefficients created by $\Phi$.

\subsection{Proof of Theorem~\ref{result}}
\begin{proof}
We now account for returns after leaving $U'$. Define the entrance time
$S_2=\inf\{t\geq0:B_t\in\mathscr C_2\}$ for the full reflected motion, and put
\[
 \rho=\sup_{z\in\mathscr C_1}\mathbb P_z(S_2<T_{\ab}).
\]
Then $\rho<1$. Indeed, every point of $\mathscr C_1$ can reach a relatively open piece of $\ab$ along a path avoiding the closed inner chart half-ball. Paths that enter this half-ball can be rerouted through the connected chart half-annulus, since $n\geq2$. The support property of the reflected diffusion gives a positive probability of absorption before return to $\mathscr C_2$. This probability is continuous on the compact set $\mathscr C_1$ (including its reflecting endpoints), by local elliptic boundary regularity, and hence has a positive minimum. This proves the strict bound, with $\rho$ independent of $\epsilon$.

Set $M_\epsilon=\sup_{x\in\mathscr C_2}u_\epsilon(x)$. Before it hits the target, a path starting on $\mathscr C_1$ must return to $\mathscr C_2$, so the strong Markov property gives
\[
 \sup_{z\in\mathscr C_1}u_\epsilon(z)\leq\rho M_\epsilon.
\]
Starting at $x\in\mathscr C_2$ and stopping at $T_\epsilon\wedge S_1$ gives
\[
 q_\epsilon(x)\leq u_\epsilon(x)
 \leq q_\epsilon(x)+\rho M_\epsilon.
\]
Consequently
\begin{equation}\label{returnbound}
 \inf_{\mathscr C_2}q_\epsilon
 \leq\inf_{\mathscr C_2}u_\epsilon
 \leq M_\epsilon
 \leq\frac{\sup_{\mathscr C_2}q_\epsilon}{1-\rho}.
\end{equation}
The factor $(1-\rho)^{-1}$ accounts for all successive excursions; no universal factor $2$ is needed.

Finally, every successful path from $x_0\notin\overline{U'}$ crosses $\mathscr C_2$. With $\alpha=\mathbb P_{x_0}(S_2<T_{\ab})>0$, the correctly stopped Markov identity is
\begin{equation}\label{eq2}
 u_\epsilon(x_0)=\mathbb E_{x_0}
 \left[\mathbf1_{\{S_2<T_{\ab}\}}u_\epsilon(B_{S_2})\right].
\end{equation}
Here $\alpha$ is positive by connectedness and interior non-degeneracy, and is independent of $\epsilon$. Equations~\eqref{returnbound}, \eqref{eq2} and Lemma~\ref{localestimate} give
\[
 c\alpha b_n(\epsilon)\leq u_\epsilon(x_0)
 \leq\frac{C\alpha}{1-\rho}b_n(\epsilon).
\]
This proves \eqref{comparison} and Theorem~\ref{result}.
\end{proof}


\begin{thebibliography}{99}

\bibitem{anderson76}
Robert~F Anderson and Steven Orey, \emph{Small random perturbation of dynamical
  systems with reflecting boundary}, Nagoya Mathematical Journal \textbf{60}
  (1976), 189--216.

\bibitem{semimartingale}
Richard~F. Bass and Pei Hsu, \emph{The semimartingale structure of reflecting
  Brownian motion}, Proceedings of the American Mathematical Society
  \textbf{108} (1990), no.~4, 1007--1010.

\bibitem{BassHsu1991}
Richard F. Bass and Pei Hsu, \emph{Some potential theory for reflecting Brownian motion in H\"older and Lipschitz domains}, The Annals of Probability \textbf{19} (1991), no.~2, 486--508.

\bibitem{conformalmapsandgeo}
Dmitry Beliaev, \emph{Conformal maps and geometry}, World Scientific (Europe),
  2019.

\bibitem{Bras_2023}
Pierre Bras and Arturo Kohatsu-Higa, \emph{Simulation of reflected Brownian
  motion on two dimensional wedges}, Stochastic Processes and their
  Applications \textbf{156} (2023), 349--378.

\bibitem{Burdzy_2013}
Krzysztof Burdzy and Zhen-Qing Chen, \emph{Reflecting random walk in fractal
  domains}, The Annals of Probability \textbf{41} (2013), no.~4, 2791--2819.

\bibitem{Chaigneau_2023}
Adrien Chaigneau and Denis~S Grebenkov, \emph{Effects of target anisotropy on
  harmonic measure and mean first-passage time}, Journal of Physics A:
  Mathematical and Theoretical \textbf{56} (2023), no.~23, 235202.

\bibitem{FukushimaOshimaTakeda2011}
Masatoshi Fukushima, Yoichi Oshima and Masayoshi Takeda, \emph{Dirichlet forms and symmetric Markov processes}, second revised and extended edition, De Gruyter Studies in Mathematics, vol.~19, Walter de Gruyter, Berlin, 2011.

\bibitem{10.3150/20-BEJ1257}
Peter Grandits, \emph{{Asymptotics of the hitting probability for a small
  sphere and a two dimensional Brownian motion with discontinuous anisotropic
  drift}}, Bernoulli \textbf{27} (2021), no.~2, 853--865.

\bibitem{PhysRevE.105.054107}
Denis~S. Grebenkov and Alexei~T. Skvortsov, \emph{Mean first-passage time to a
  small absorbing target in three-dimensional elongated domains}, Phys. Rev. E
  \textbf{105} (2022), 054107.

\bibitem{HolcmanSchuss2014}
David Holcman and Zeev Schuss, \emph{The narrow escape problem}, SIAM Review \textbf{56} (2014), no.~2, 213--257, \href{https://doi.org/10.1137/120898395}{doi:10.1137/120898395}.

\bibitem{hittingdensity}
Qin Hu, Yongjin Wang, and Xuewei Yang, \emph{The hitting time density for a
  reflected Brownian motion}, Computational Economics \textbf{40} (2012),
  1--18.

\bibitem{LaknerLiuReed2023RBMwedge}
Peter Lakner, Ziran Liu, and Josh Reed, \emph{Reflected Brownian motion with
  drift in a wedge}, Queueing Systems \textbf{105} (2023), no.~3-4, 233--270.

\bibitem{lawler08}
Gregory~F Lawler, \emph{Conformally invariant processes in the plane}, no. 114,
  American Mathematical Society, 2008.

\bibitem{doi:10.1137/16M1077659}
Alan~E. Lindsay, Andrew~J. Bernoff, and Michael~J. Ward, \emph{First passage
  statistics for the capture of a Brownian particle by a structured spherical
  target with multiple surface traps}, Multiscale Modeling \& Simulation
  \textbf{15} (2017), no.~1, 74--109.

\bibitem{LionsSznitman1984}
Pierre-Louis Lions and Alain-Sol Sznitman, \emph{Stochastic differential equations with reflecting boundary conditions}, Communications on Pure and Applied Mathematics \textbf{37} (1984), no.~4, 511--537, \href{https://doi.org/10.1002/cpa.3160370408}{doi:10.1002/cpa.3160370408}.

\bibitem{oksendal1981brownian}
Bernt {\O}ksendal, \emph{Brownian motion and sets of harmonic measure zero},
  Pacific Journal of Mathematics \textbf{95} (1981), no.~1, 179--192.

\bibitem{DLMF}
F. W. J. Olver et al. (eds.), \emph{NIST Digital Library of Mathematical Functions}, \S19.12, \url{https://dlmf.nist.gov/19.12}.

\bibitem{redner_2001}
Sidney Redner, \emph{A guide to first-passage processes}, Cambridge University
  Press, 2001.

\bibitem{SingerSchussHolcman2008}
Amit Singer, Zeev Schuss and David Holcman, \emph{Narrow escape and leakage of Brownian particles}, Physical Review E \textbf{78} (2008), 051111, \href{https://doi.org/10.1103/PhysRevE.78.051111}{doi:10.1103/PhysRevE.78.051111}.

\end{thebibliography}
\end{document}